\documentclass[10pt,one side,emlines]{amsart}
\usepackage{amssymb,latexsym,xy,eucal,mathrsfs}
\usepackage{tikz}
\usetikzlibrary{arrows.meta}
\usepackage{graphicx, graphics}
\usepackage{caption}
\usepackage{wrapfig}
\newtheorem{lem}{Lemma}[section]

\newtheorem{theorem}[lem]{Theorem}
\newtheorem{cnj}[lem]{Conjecture}
\newtheorem*{theorem*}{Main Theorem}
\newtheorem{thm}[lem]{Theorem}

\theoremstyle{definition}

\newtheorem*{cthm*}{Conjecture}

\newcommand{\ie}{${\it{i.e.,}}$ }

\begin{document}
	\baselineskip 15truept
	
	\subjclass[2020]{Primary 05C15, Secondary 06A12, 13A70} %
	
	\title{ A Proof of the Conjecture on complemented zero-divisor graphs of semigroups}
	\maketitle
	\markboth{Anagha Khiste, Ganesh Tarte and Vinayak Joshi}{A proof of the conjecture-3}\begin{center}\begin{large} $^\text{a}$Anagha Khiste, $^\text{b}$Ganesh Tarte and $^\text{c}$Vinayak Joshi \end{large}\begin{small}\vskip.1in$^\text{a}$\emph{Department of Applied Science and Humanities, Indian Institute of Information Technology, Pune - 410507 }\\$^\text{b}$\emph{Department of Applied Sciences and Humanities, Pimpri Chinchwad College of Engineering, Pune - 411044}\\$^\text{c}$\emph{Department of Mathematics, Savitribai Phule Pune University, Pune - 411007, Maharashtra, India\\E-mail: avanikhiste@gmail.com, ganesh.tarte@pccoepune.org, vvjoshi@unipune.ac.in, vinayakjoshi111@yahoo.com }\end{small}\end{center}\vskip.2in

\begin{abstract}
		In this paper, we are motivated by the conjectures proposed by C.~Bender \textit{et al.}, \cite{C} in 2024. We have settled the first two conjectures negatively by providing  a counter example in \cite{KTJ}, whereas in this paper, we prove the third conjecture positively, which has remained an open question until now.
		The third conjecture is stated as if $G(S)$ is uniquely complemented with the  clique number $3$ or greater and has the property that every vertex has a unique complement, then the graph $G(S)$ is isomorphic to the graph $G(\mathcal{P}(n))$, where $n$ is the clique number of $G(S)$.

	\end{abstract}
	
	\maketitle \noindent{ \small \textbf{Keywords}: Clique number, commutative semi-group,  complemented graph, reduced graph, uniquely complemented graph, zero-divisor graph.} 
	 
	\maketitle
	
	\section{Introduction}
	The zero-divisor graph associated with a commutative ring was first introduced by I. Beck \cite{B} with the goal of studying graph coloring of rings. In Beck’s construction, the vertex set of the zero-divisor graph is given by the set of all elements of the ring and two vertices are adjacent in the graph if their product is	zero. Anderson-Livingston \cite{AL} modified this construction so the vertex set corresponds to only elements of the ring which are nonzero zero-divisors. This construction is generalized to the zero-divisor graph associated with a semigroup in \cite{DD,DMS}. In recent years, there has been significant progress on the study of the zero-divisor graph associated with a semigroup.

	Throughout, let $S$ denote a commutative semigroup with $0$. The set of zero-divisors of $S$, denoted $Z(S)$, forms an ideal in $S$. The zero-divisor graph associated with $S$, denoted $G(S)$, is the graph whose vertices are given by the nonzero zero-divisors and where two distinct vertices $a$ and $b$ are	adjacent precisely when $ab = 0$.
	
	In a graph $G$, denote the vertex set of the graph by $V(G)$. The \textit{neighborhood} of vertex $a \in V(G)$, denoted by $N(a)$, is the set of vertices adjacent to $a$. Define $\overline{N(a)} = N(a) \cup \{a\}$. The degree $deg(a)$ of a vertex $a$ in
	$G$ is the number of edges incident to $a$, \textit{i.e.,} $deg(a)=|N(a)|$. A \textit{complete graph} is a graph	where every pair of distinct vertices is adjacent. A vertex of degree 1 is called an \textit{end}. A complete graph on $n$ vertices is denoted $K_n$.  A set $I$ of vertices of a graph $G$ is said
	to be independent if no two vertices $u$ and $v$ in $I$ are adjacent in $G$. If $|I| = n$, then we denote $I_n$, the independent set on n vertices. A \textit{clique} in a graph $G$ is a sub-graph of $G$ where every pair of distinct vertices in the sub-graph is adjacent. The	\textit{clique number} of a graph $G$ is the number of vertices in a maximal clique in $G$ and is denoted $\omega(G)$. If	a graph $G$ contains a cycle, then the core is defined to be the subgraph induced on all vertices which are contained in any cycle. If the graph $G$ does not contain any cycle, then the core of $G$ is empty. Two distinct vertices $a$ and $b$ are called \textit{orthogonal}, denoted $a\perp b$, if $a$ and $b$ are adjacent and there does not exist a vertex $c$ adjacent to both $a$ and $b$. Equivalently, adjacent vertices $a$ and $b$ are \textit{orthogonal} if edge $a - b$ is not an edge of any triangle (3-cycle). A graph $G$ is called \textit{complemented} if for every vertex $a$, there exists a vertex $b$ such that $a\perp b$. A graph $G$ is called \textit{uniquely complemented} if $G$ is complemented and has the property that if $a, b, c$ are distinct vertices with $a\perp b$ and $a\perp c$ then $N(b) = N(c)$.  This graph-theoretic property has been studied in the context of the zero-divisor graph associated with a commutative ring, for example in \cite{L}.

	Let $\mathcal{P}(n)$ denote the semigroup given by the power set of a set of $n$ elements under the operation of intersection.
	The author's C. Bender et al. \cite{C} raised the following three conjectures for a commutative semigroup $S$ with 0.
	
	\begin{cnj}\label{c1} If $G(S)$ is a complemented zero-divisor graph with clique number $n \geq 3$, then $G(S)$ has reduced graph $G_r$ which is isomorphic to the graph $G(\mathcal{P}(n))$.
	\end{cnj}
	\begin{cnj} \label{c2} If $G(S)$ is a complemented zero-divisor graph with clique number $3$ or greater, then $G(S)$ is uniquely complemented.
	\end{cnj}
		 \begin{cnj}\label{c3}
	If $G(S)$ is uniquely complemented with clique number $3$ or greater and has the property that every vertex has a unique complement, then graph $G(S)$ is isomorphic to graph $G(\mathcal{P}(n))$ where $n$ is the clique number of $G(S)$.
	\end{cnj} 

We settled the first two Conjectures \ref{c1} and \ref{c2} negatively by giving the counter example; (\textit{see} \cite{KTJ}), and now answering  Conjecture~\ref{c3} positively in the next section as  Theorem \ref{2.2} .
	
	\section{A Proof of the Conjecture-3}
	
	In \cite{DD}, the following results are established for the zero-divisor graph associated with a commutative	semigroup $S$ .
	\begin{thm}\label{2.1} Let $G(S)$ denote the zero-divisor graph associated with a commutative semigroup $S$.
	Then $G(S)$ has the following properties.
\begin{enumerate}\item The graph $G(S)$ is connected.
\item $G(S)$ has diameter $\leq 3$.
\item If $G(S)$ contains a cycle, then the graph $G(S)$ consists of a core, which is a union of quadrilaterals and	triangles, and any vertex not in the core of $G(S)$ is a vertex of degree 1.
	\item For each pair $x$ and $y$ of non-adjacent vertices, there is a vertex $z$ with $N(x)\cup N(y)\subseteq \overline{N(z)}$.
\end{enumerate}\end{thm}
	Now, we  prove  Conjecture~\ref{c3}.
	\begin{theorem}\label{2.2}
		Let $S$ be a commutative semigroup with $0$. Let $G(S)$ be the uniquely complemented graph such that every vertex has a unique complement. Let $\omega(G(S))=n \geq 3$. Then $G(S) \simeq G(\mathcal{P}(n))$, ({\it{i.e.,}} $G(S) \simeq G(\mathcal{P}(X))$, where $X=\{x_1,x_2, \cdots x_n\})$.  Furthermore, if $S=V(G(S))\cup \{0,1\}$, then $S \simeq \mathcal{P}(X) (=\mathcal{P}(n))$
	\end{theorem}

\begin{proof}
		Let $(S,*)$ be a commutative semigroup with $0$. Let $G(S)$ be the uniquely complemented graph where every vertex has the unique complement and let $\omega(G(S))=n \geq 3$. Thus, there exists a maximal clique, say $\{a_1,a_2,\cdots,a_n\}$ in $G(S)$. \\ 
	
		For better understanding, we present the outline of the proof for the cases 
		$n=3$ and $n=4$.\\
		\textbf{For $n=3$:} Let $\{a_1,a_2,a_3\}$ be a clique in $G(S)$. Since every vertex has the unique complement in $G(S)$, for each $a_i$,  there exists the unique $b_i(\neq a_i)\in G(S)$ such that $a_i \perp b_i$. 
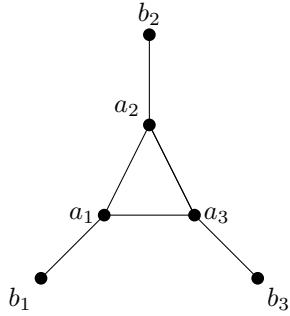
		\begin{center}
			\begin{tikzpicture}[scale=1.2]
				\coordinate (A) at (-2,0);
				\coordinate (B) at (-1.5,1);
				\coordinate (C) at (-1,0);
				\coordinate (D) at (-0.3,-0.7);
				\coordinate (E) at (-1.5,2);
				\coordinate (F) at (-2.7,-0.7);
				\draw (A) -- (B) -- (C) -- cycle;
				
				\draw (B) -- (E); \draw (A) -- (F);
				\draw (B) -- (C) -- (D) ;
				\fill (A) circle (2pt);
				\fill (B) circle (2pt);
				\fill (C) circle (2pt);
				\fill (D) circle (2pt);
				\fill (E) circle (2pt);
				\fill (F) circle (2pt);
				\node[left] at (A) {$a_1$};
				\node[below][right] at (C) {$a_3$};
				\node[above left] at (B) {$a_2$};
				\node[below right] at (D) {$b_3$};
				\node[above] at (E) {$b_2$};
				\node[below left] at (F) {$b_1$};
			\end{tikzpicture}
		\end{center}
		We claim that the family of sets $\{A_k~/~1\leq k\leq 2\}$, where $A_1=\{b_1, b_2,b_3\}$ and $A_2=\{a_1=b_2*b_3,a_2=b_1*b_3,a_3=b_1*b_2\}$ forms a partition of the vertex set $V(G(S))$. Moreover, we can show that $x^2=x,$ for every $x\in V(G(S))=A_1\cup A_2$ and $deg(x)=2^{k}-1$, for every $x\in V(G(S))$. Thus, for each $b_i\in A_1$, $deg(b_i)=1=2^1-1=2^k-1$ and $deg(a_i)=3=2^2-1=2^{3-1}-1=2^{n-1}-1$ and . Also, $G[A_1]\cong I_n$ and $G[A_2]\cong K_n$.
		In fact we obtain the binary operation on $V(G(S))$ as follows.
		
		\begin{center}
			\begin{tabular}{c|c c c c c c}
				$*$ & $a_1$ & $a_2$ & $a_3$ & $b_1$ & $b_2$ & $b_3$ \\ 
				\hline
				$a_1$ & $a_1$ & $0$ & $0$ & $0$ & $a_1$ & $a_1$ \\
				$a_2$ & $0$ & $a_2$ & $0$ & $a_2$ & $0$ & $a_2$ \\
				$a_3$ & $0$ & $0$ & $a_3$ & $a_3$ & $a_3$ & $0$ \\
				$b_1$ & $0$ & $a_2$ & $a_3$ & $b_1$ & $a_3$ & $a_2$ \\
				$b_2$ & $a_1$ & $0$ & $a_3$ & $a_3$ & $b_2$ & $a_1$ \\
				$b_3$ & $a_1$ & $a_2$ & $0$ & $a_2$ & $a_1$ & $b_3$
			\end{tabular}
		\end{center}
		Thus, we observe that, $b_1*b_2*b_3=0$ and $a_i(=b_j*b_k) \perp b_i, \text{for all distinct } b_i,b_j,b_k \in A_1.$ 
		Further, $|V(G(S))|=|A_1|+|A_2|=3+3=6=2^3-2=2^n-2=|V(G(\mathcal{P}(X)))|$, where $X=\{x_1,x_2,x_3\}$. Thus we define a bijective map $f:V(G(S))\rightarrow V(G(\mathcal{P}(X)))$ as $f(a_i)=\{x_i\},$ for all $a_i\in A_2=A_{n-1}$ and $f(b_i)=\{x_j,x_k\}=X\setminus \{x_i\}$, for all $b_i\in A_1$, where $i,j,k\in \{1,2,3\}$ are all distinct integers. Then we observe that $f$ is a graph isomorphism and hence $G(S)\cong G(\mathcal{P}(X))$.

		\vspace{1cm}

		\textbf{For $n=4$:} Let $\{a_1,a_2,a_3,a_4\}$ be a clique in $G(S)$. Since every vertex has the unique complement in $G(S)$, for each $a_i$,  there exists the unique $b_i(\neq a_i)\in G(S)$ such that $a_i \perp b_i$. 
		We claim that the family of sets $\{A_k~/~1\leq k\leq 3\}$, where $A_1=\{b_1, b_2, b_3\}$, $A_2=\{b_1*b_2,b_1*b_3,b_1*b_4,b_2*b_3,b_2*b_4, b_3*b_4\}$ and $A_3=\{a_1=b_2*b_3*b_4,a_2=b_1*b_3*b_4,a_3=b_1*b_2*b_4,a_4=b_1*b_2*b_3\}$ forms a partition of the vertex set $V(G(S))$. Moreover, we can show that $x^2=x,$ for every $x\in V(G(S))=A_1\cup A_2\cup A_3$ and $deg(x)=2^{k}-1$, for every $x\in V(G(S))$. Thus for each distinct $b_i,b_j\in A_1 $,  $deg(b_i)=1=2^1-1=2^k-1$, $deg(b_i*b_j)=3=2^2-1=2^{k}-1$ $deg(a_i)=7=2^3-1=2^{4-1}-1=2^{n-1}-1$ Also, $G[A_1]\cong I_n$ and $G[A_3]\cong K_n$.
		In fact we obtain the binary operation on $V(G(S))$ as follows.
		
		\begin{center}
			$\begin{array}{c|c|c|c|c|c|c|c|c|c|c|c|c|c|c|}
			
				\ast & a_1 & a_2 & a_3 & a_4 & b_1 \ast b_2 & b_1 \ast b_3 & b_1 \ast b_4 & b_2 \ast b_3 & b_2 \ast b_4 & b_3 \ast b_4 & b_1 & b_2 & b_3 & b_4  \\
				\hline
				a_1 & a_1 & 0 & 0 & 0 & a_1 & a_1 & a_1 & 0 & 0 & 0 & a_1 & a_1 & a_1 & 0  \\
				
				a_2 & 0 & a_2 & 0 & 0 & a_2 & 0 & 0 & a_2 & a_2 & 0 & a_2 & a_2 & 0 & a_2  \\
				
				a_3 & 0 & 0 & a_3 & 0 & 0 & a_3 & 0 & a_3 & 0 & a_3 & a_3 & 0 & a_3 & a_3  \\
				
				a_4 & 0 & 0 & 0 & a_4 & 0 & 0 & a_4 & 0 & a_4 & a_4 & 0 & a_4 & a_4 & a_4  \\
				
				b_1 \ast b_2  & 0 & 0 & a_3 & a_4 & 0 & a_3 & a_4 & a_3 & a_4 & b_1 \ast b_2 & a_3 & a_4 & b_1 \ast b_2 & b_1 \ast b_2 \\
				
				b_1 \ast b_3 & 0 & a_2 & 0 & a_4 & a_2 & 0 & a_4 & a_2 & b_1 \ast b_3 & a_4 & a_2 & b_1 \ast b_3 & a_4 & b_1 \ast b_3  \\
				
				b_1 \ast b_4  & 0 & a_2 & a_3 & 0 & a_2 & a_3 & 0 & b_1 \ast b_4 & a_2 & a_3 & b_1 \ast b_4 & a_2 & a_3 & b_1 \ast b_4  \\
				
				b_2 \ast b_3 & a_1 & 0 & 0 & a_4 & a_1 & a_1 & b_2 \ast b_3 & 0 & a_4 & a_4 & a_1 & b_2 \ast b_3 & b_2 \ast b_3 & a_4  \\
				
				b_2 \ast b_4 &  a_1 & 0 & a_3 & 0 & a_1 & b_2 \ast b_4 & a_1 & a_3 & 0 & a_3 & b_2 \ast b_4 & a_1 & b_2 \ast b_4 & a_3  \\
				
				b_3 \ast b_4 &  a_1 & a_2 & 0 & 0 & b_3 \ast b_4 & a_1 & a_1 & a_2 & a_2 & 0 & b_3 \ast b_4 & b_3 \ast b_4 & a_1 & a_2  \\
				
				b_1 & 0  & a_2 & a_3 & a_4 & a_2 & a_3 & a_4 & b_1 \ast b_4 & b_1 \ast b_3 & b_1 \ast b_2 & b_1 \ast b_4 & b_1 \ast b_3 & b_1 \ast b_2 & b_1  \\
				
				b_2 &  a_1 & 0 & a_3 & a_4 & a_1 & b_2 \ast b_4 & b_2 \ast b_3 & a_3 & a_4 & b_1 \ast b_2 & b_2 \ast b_4 & b_2 \ast b_3 & b_2 & b_1 \ast b_2  \\
				
				b_3 &  a_1 & a_2 & 0 & a_4 & b_3 \ast b_4 & a_1 & b_2 \ast b_3 & a_2 & b_1 \ast b_3 & a_4 & b_3 \ast b_4 & b_3 & b_2 \ast b_3 & b_1 \ast b_3  \\
				
				b_4 & a_1 & a_2 & a_3 & 0 & b_3 \ast b_4 & b_2 \ast b_4 & a_1 & b_1 \ast b_4 & a_2 & a_3 & b_4 & b_3 \ast b_4 & b_2 \ast b_4 & b_1 \ast b_4 	\end{array}$
	
		\end{center}
		\vskip5pt
		Thus, we observed that, $b_1*b_2*b_3*b_4=0$. Also, $a_i(=b_j*b_k*b_s) \perp b_i$, and \linebreak $(b_i*b_j) \perp (b_k*b_s), \text{for all distinct } b_i,b_j,b_k, b_s \in A_1.$ 
		
		
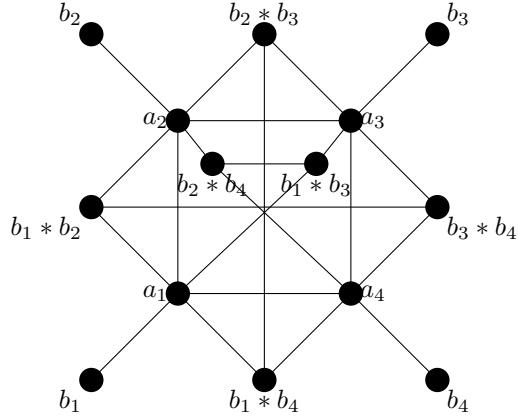
\begin{center}
			\begin{tikzpicture}[scale=2.3]
				\coordinate (A) at (0,0);
				\coordinate (B) at (0,1);
				\coordinate (C) at (1,1);
				\coordinate (D) at (1,0);
				\coordinate (E) at (-0.5,1.5);
				\coordinate (F) at (1.5,1.5);
				\coordinate (G) at (1.5,-0.5);
				\coordinate (H) at (-0.5,-0.5);
				\coordinate (I) at (0.5,1.5);
				\coordinate (J) at (1.5,0.5);
				\coordinate (K) at (0.5,-0.5);
				\coordinate (L) at (-0.5,0.5);
				\coordinate (M) at (0.2,0.75);
				\coordinate (N) at (0.8,0.75);
				
				\draw (A) -- (B) -- (C) -- (D) --(A);
				
				\draw (A) -- (H); \draw (B) -- (E);
				\draw (C) -- (F); \draw (D) -- (G);
				\draw (B) -- (I) -- (C);
				\draw (C) -- (J) -- (D); 
				\draw (D) -- (K) -- (A);
				\draw (A) -- (L) -- (B);
				
				\draw (B) -- (M) -- (D);
				\draw (A) -- (N) -- (C); 
				
				\draw (K) -- (I);
				\draw (L) -- (J); 
				\draw (M) -- (N);
				
				\fill (A) circle (2pt);
				\fill (B) circle (2pt);
				\fill (C) circle (2pt);
				\fill (D) circle (2pt);
				\fill (E) circle (2pt);
				\fill (F) circle (2pt);
				\fill (G) circle (2pt);
				\fill (H) circle (2pt);
				\fill (I) circle (2pt);
				\fill (J) circle (2pt);
				\fill (K) circle (2pt);
				\fill (L) circle (2pt);
				\fill (M) circle (2pt);
				\fill (N) circle (2pt);
				
				\node[left] at (A) {$a_1$};
				\node[left] at (B) {$a_2$};
				\node[right] at (C) {$a_3$};
				\node[right] at (D) {$a_4$};
				\node[above left] at (E) {$b_2$};
				\node[above right] at (F) {$b_3$};
				\node[below right] at (G) {$b_4$};
				\node[below left] at (H) {$b_1$};
				\node[above] at (I) {$b_2 \ast b_3$};
				\node[below right] at (J) {$b_3 \ast b_4$};
				\node[below] at (K) {$b_1 \ast b_4$};
				\node[below left] at (L) {$b_1 \ast b_2$};
				\node[below] at (M) {$b_2 \ast b_4$};
				\node[below] at (N) {$b_1 \ast b_3$};
			\end{tikzpicture}
		\end{center}
		
		Further, $|V(G(S))|=|A_1|+|A_2|+|A_3|=4+6+4=14=2^4-2=2^n-2=|V(G(\mathcal{P}(X)))|$, where $X=\{x_1,x_2,x_3,x_4\}$. Thus we define a bijective map $f:V(G(S))\rightarrow V(G(\mathcal{P}(X)))$ as $f(a_i)=\{x_i\},$ for all $a_i\in A_3=A_{n-1}$, $f(b_i)=\{x_j,x_k,x_s\}=X\setminus \{x_i\}$, and $f(b_i*b_j)=\{x_k,x_s\}=X\setminus \{x_i,x_j\}$, for all distinct $b_i,b_j\in A_1$, where $i,j,k,s\in \{1,2,3,4\}$ are all distinct integers. Then we observe that $f$ is a graph isomorphism and hence $G(S)\cong G(\mathcal{P}(X))$.
	
	With this preparation, we prove the result for general $n$.
		
		\textbf{Now, for each $1\leq i \neq j\leq n$, we have the following observations:}

	\begin{enumerate}
		\item Clearly, $a_i$ is adjacent to $a_j$,  {\it{i.e.,}}   $a_i*a_j=0$. Thus, $G[\{a_i~|~1\leq i\leq n\}]\cong K_n$ with $n\geq 3$, where $G[H]$ denotes the induced subgraph generated by $H\subseteq V(G(S))$ This implies that the edge $a_i-a_j$ is an edge of a triangle. Thus, $a_i\not\perp a_j$, for all $i\neq j$, $1\leq i,j\leq n$.
		
		\item Since every vertex has the unique complement in $G(S)$, for each $a_i$,  there exists the unique $b_i(\neq a_i)\in G(S)$ such that $a_i \perp b_i$. Now, if $b_i=b_j,$ for some $i\neq j,$ $1\leq i,j\leq n$, then $b_i$ has two distinct complements namely, $a_i$ and $a_j$, which is not possible. Therefore, $b_i\neq b_j$, for all $i\neq j$, $1\leq i, j\leq n$. Further, if $b_i =a_j$ for some $i\neq j$, then $a_i\perp b_i$ gives that $a_i\perp a_j$, which is not possible. 		
	Also, $a_i$ and $b_j$ are not adjacent to each other, otherwise $a_i-b_j-a_j-a_i$ forms a triangle in $G(S)$, which contradicts to $a_j \perp b_j$. Hence \[a_i*b_j \neq 0, \quad \text{for all}~ i \neq j, \quad 1\leq i,j\leq n.   \hfill ----(A)\] 
	
	\item Now, we claim that $\{b_1,b_2,\cdots,b_n\}$ is an independent set. \\ Suppose, $b_i$ and $b_j$ are adjacent, for some $i\neq j$, where $1\leq i, j\leq n$.\\ Then we have the following cases:
		\medskip
	
	\textbf{Case (i):}  If the edge $b_i-b_j$ is \textit{not} an edge of any triangle, then $b_i \perp b_j$. Thus, $b_i$ would have two distinct complement namely, $a_i$ and $b_j$, which contradicts the assumption that every vertex has the unique complement. 
		\medskip
		
		\textbf{Case (ii):}  Suppose the edge $b_i-b_j$ is an edge of a triangle, say $x-b_i-b_j-x$  in $G(S)$. Then $x \notin \{b_i,b_j\}$, otherwise it is not a triangle. Since $a_i \perp b_i$ and $a_j \perp b_j$, and using equation $(A)$, we deduce that $x\notin \{b_i,b_j,a_k~|~1 \leq k \leq n\}$. Also, we have $x$ is not adjacent to $a_i$ as well as $a_j$ in $G(S)$. Further, consider the set 
		 \[
		 \{b_i, b_j, a_1,a_2,\cdots,a_{i-1},a_{i+1},\cdots,a_n\}=\{b_i, b_j, a_k\mid 1\leq k \neq i\leq n\} \subseteq N(x) \cup N(a_i).
		 \]
		 
		  Applying Theorem \ref{2.1} for $x$ and $a_i$, there exists a vertex $z\in V(G(S))$ such that \[N(x) \cup N(a_i) \subseteq \overline{N(z)}=N(z) \cup \{z\}.\] Therefore,\[\{b_i,b_j,a_k \mid 1\leq k \neq i\leq n\} \subseteq \overline{N(z)}=N(z) \cup \{z\}.\] Now, consider the following possibilities:
			\begin{itemize}
				\item If $z=a_j$, then $b_i \in N(a_j)$, which contradicts equation $(A)$. Thus $z\neq a_j$. 
			
		\item If $z=b_j$, then $b_i \neq z$ and hence $a_p \in N(b_j)$, for $1\leq p \neq i,j\leq n$, again contradicting equation $(A)$. Thus $b_j \neq z$.
			
		\item	If $z\not \in \{a_j, ~b_j\}$, but $a_j,b_j \in N(z)$ forms a triangle in $G(S)$, which contradicts $a_j \perp b_j$. Thus $b_j \neq z$.
	\end{itemize}
			Thus, there does not exist any vertex $x$, such that $x - b_i - b_j - x$ forms a triangle.

			Therefore, if $b_i$ and $b_j$ are adjacent for some $i \neq j$, $1\leq i, j\leq n$, in each case, we arrive at a contradiction. Hence $b_i$ and $b_j$ are non-adjacent for all $i \neq j$, $1\leq i, j\leq n$. That is \[b_i*b_j\neq 0, \quad \text{for all } i \neq j,  1\leq i, j\leq n. \hfill ----(B)  \] 
			\smallskip
Thus, $\{b_1,b_2,\cdots,b_n\}$ is an independent set.	Thus $G[\{b_i~|~1\leq i\leq n\}]\cong I_n$.  \\		

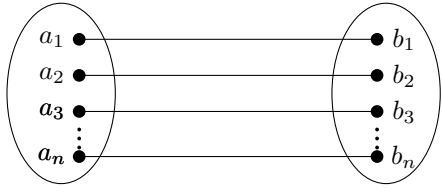
\begin{center}
	\begin{tikzpicture}[scale=1.2]

		\draw (-2.6,3) ellipse (0.6 and 1);
		\node at (-2.7,3.6) {$a_1$};
		\node at (-2.7,3.2) {$a_2$};
		\node at (-2.7,2.8) {$a_3$};
		\node at (-2.7,2.3) {$a_n$};
		\node at (-2.7,2.8) {$a_3$};
		\node at (-2.7,2.3) {$a_n$};								
		
		\draw (0.9,3.6) -- (-2.4,3.6);
		\draw (0.9,3.2) -- (-2.4,3.2);
		\draw (0.9,2.8) -- (-2.4,2.8);
		\draw (0.9,2.3) -- (-2.4,2.3);
		
		\fill (-2.4,3.6) circle (2pt);
		\fill (-2.4,3.2) circle (2pt);
		\fill (-2.4,2.8) circle (2pt);
		\fill (-2.4,2.6) circle (.6pt);
		\fill (-2.4,2.5) circle (.6pt);
		\fill (-2.4,2.4) circle (.6pt);
		\fill (-2.4,2.3) circle (2pt);
		\draw (1,3) ellipse (0.6 and 1);
		\node at (1.2,3.6) {$b_1$};
		\node at (1.2,3.2) {$b_2$};
		\node at (1.2,2.8) {$b_3$};
		\node at (1.2,2.3) {$b_n$};
		
		\fill (0.9,3.6) circle (2pt);
		\fill (0.9,3.2) circle (2pt);
		\fill (0.9,2.8) circle (2pt);
		\fill (0.9,2.6) circle (.6pt);
		\fill (0.9,2.5) circle (.6pt);
		\fill (0.9,2.4) circle (.6pt);
		\fill (0.9,2.3) circle (2pt);
		
		\node at (-2.5,1.5) {$G[\{a_i~|~1\leq i\leq n\}]\cong K_n$};
		\node at (1.5,1.5) {$G[\{b_i~|~1\leq i\leq n\}]\cong I_n$};
	\end{tikzpicture}
\end{center}

Now, since, $a_i \in N(b_i)$, for all $i$, $1\leq i\leq n$ and $b_i,~a_j \in N(a_i)$, for all $j \neq i$, $1\leq j, i\leq n$, we have, \[deg(b_i) \geq 1\quad \text{and}\quad deg(a_i) \geq n.\]

\item Now we claim that, $deg(b_i)=1,~~ \text{for all } i,~~ 1\leq i\leq n$. \\Suppose $deg(b_i)>1,$ for some $i$, $1\leq i\leq n$. Then there exist $x_i \in V(G(S)) \setminus \{a_i, b_i\}$ such that $x_i$ is adjacent to $b_i$. Together with equations $(A)$ and $(B)$, we deduce that \[x_i \not \in\{a_k, ~b_k~|~ 1 \leq k \leq n\}.\]

Now, we consider the following cases;
\medskip

\textbf{Case (i):}  If the edge $x_i-b_i$ is \textit{not} an edge of any triangle, then $x_i \perp b_i$. Thus $b_i$ would have two distinct complements namely, $a_i$ and $x_i$, which contradicts the fact that every vertex has the unique complement. 
\medskip

\textbf{Case (ii):}  Suppose the edge $x_i-b_i$ is an edge of a triangle, say $x_i-b_i-p_i-x_i$  in $G(S)$.
Clearly, $p_i\notin \{x_i, b_i\}$. Also, since $a_i\perp b_i$ and using equation $(A)$, we deduce that $p_i\notin \{a_k ~|~ 1\leq k\leq n\}$.

Now, since $\omega(G(S))=n$, the set $\{p_i,x_i,a_k ~|~ 1\leq k \neq i\leq n\}$, which contains $n+1$ vertices does not form a clique. 
Together with the fact that $p_i$ is adjacent to $x_i$ and $G[\{a_k ~|~ 1\leq k \neq i\leq n\}] \cong K_{n-1}$, this implies that either \[
x_i \notin \bigcap_{\substack{k=1 \\ k \ne i}}^{n} N(a_k) \quad \text{or} \quad p_i \notin \bigcap_{\substack{k=1 \\ k \neq i}}^{n} N(a_k).
\]
Assume that $x_i \notin \displaystyle \bigcap_{\substack{k=1 \\ k \ne i}}^{n} N(a_k).
$ Thus, we have $x_i \notin N(a_k)$, for some $k \neq i$, $1\leq k, i\leq n$. 

Applying Theorem \ref{2.1} for $x_i$ and $a_k$, there exists a vertex $z\in V(G(S))$ such that \[\{b_i, p_i,a_i,a_j,b_k ~|~1\leq  j (\neq ~ k,~i)\leq n\} \subseteq N(x_i) \cup N(a_k) \subseteq \overline{N(z)}=N(z) \cup \{z\}.\] Now, we consider the following possibilities:
\begin{itemize}		\item If $z=a_i$, then $b_k \in N(a_i)$, which contradicts equation $(A)$, thus $z \neq a_i$. 
	\item If $z=b_i$, then $a_j \in N(b_i)$, again  contradicting equation $(A)$, thus $z \neq b_i$.
	\item If $z \not \in \{a_i,b_i\}$, but $a_i,b_i \in N(z)$ forms a triangle in $G(S)$, which contradicts $a_i \perp b_i$.
\end{itemize}
Similarly, if we assume that $p_i \notin \displaystyle \bigcap_{\substack{k=1 \\ k \ne i}}^{n} N(a_k),$ then we have $p_i \notin N(a_k)$, for some $k \neq i$, $1\leq k, i\leq n$.

 Therefore, applying Theorem \ref{2.1} for $p_i$ and $a_k$, there exists a vertex $z\in V(G(S))$ such that \[\{b_i, x_i,a_i,a_j,b_k ~|~1\leq  j (\neq ~ k,~i)\leq n\} \subseteq N(p_i) \cup N(a_k) \subseteq \overline{N(z)}=N(z) \cup \{z\}.\] Now, we have the following possibilities:
\begin{itemize}
		\item If $z=a_i$, then $b_k \in N(a_i)$, which contradicts equation $(A)$, thus $z \neq a_i$. 
	\item If $z=b_i$, then $a_j \in N(b_i)$, again  contradicting equation $(A)$, thus $z \neq b_i$.
	\item If $z \not \in \{a_i,b_i\}$, but $a_i,b_i \in N(z)$ forms a triangle in $G(S)$, which contradicts $a_i \perp b_i$.
\end{itemize}
Therefore, if we assume that $deg(b_i)> 1$, in each case, we arrive at a contradiction. Thus, there does not exists a vertex $x_i \in V(G(S)) \setminus \{a_i, b_i\}$ such that $x_i$ is adjacent to $b_i$. This implies that $N(b_i)=\{a_i\}$, and hence  $deg(b_i)=1,$ for all $i$, $1\leq i\leq n$. That is each $b_i$ is an end vertex in $G(S)$.

\item We claim that, $a_i * b_j=a_i$, for all $i \neq j$, $1\leq i, j\leq n$. \\ Using equation $(A)$, we have $a_i * b_j \neq 0$, for any $i \neq j$, $1\leq i, j\leq n$. Then $b_i*a_i * b_j = 0$, which implies that $a_i * b_j \in V(G(S))$ such that $a_i * b_j\in N(b_i)\cup \{b_i\}$.  

If $a_i * b_j =b_i$, then $b_i*a_j=0$, which is not possible. 

Thus, $a_i * b_j\in N(b_i)=\{a_i\}$ and hence $a_i * b_j=a_i$, for all $i \neq j$, $1\leq i, j\leq n$. \hfill $---(C)$

\item We claim that $b_i^2=b_i*b_i=b_i$, for all $i$, $1\leq i\leq n$.

 Assume that $b_k^2\neq b_k$, for some $k$, 
$1\leq k\leq n$.
Using $(C)$ twice, we have $a_i*b_k^2=a_i$, for all $i \neq k$, $1\leq i,k \leq n$. Also, we have $a_k*b_k^2=a_k*b_k*b_k=0*a_k=0$.
\begin{itemize}
 \item If $b_k^2=0$, then $a_i=a_i*b_k^2=0$, which is not possible. Hence $b_k^2\neq 0$. 
 
 \item If $b_k^2=a_k$, then $a_i=a_i*b_k^2=a_i*a_k=0$ for all $i \neq k$, $1 \leq i,k \leq n$, which is also not possible. Hence $b_k^2\neq a_k$. 
 \end{itemize}
  Thus, $b_k^2\in V(G(S))\setminus \{a_k\}$ such that $b_k^2$ is adjacent to $a_k$. Together with $b_k$ is the unique complement of $a_k$ and $b_k^2\neq b_k$, we have $b_k^2\not\perp a_k$, \ie the edge $b_k^2-a_k$ is an edge of a triangle in $G(S)$.  Thus, there exist $x_k\in V(G(S))$ such that $a_k-b_k^2-x_k-a_k$ forms a triangle in $G(S)$. Further more, $a_k \perp b_k$, thus $x_k$ is not adjacent to $b_k$, \ie $x_k*b_k\neq 0$. But we have $x_k*b_k^2=0$. Therefore, $x_k*b_k\in V(G(S))$ such that $x_k*b_k*b_k=x_k*b_k^2=0$. Thus $x_k*b_k\in N(b_k)\cup \{b_k\}=\{a_k,b_k\}$.

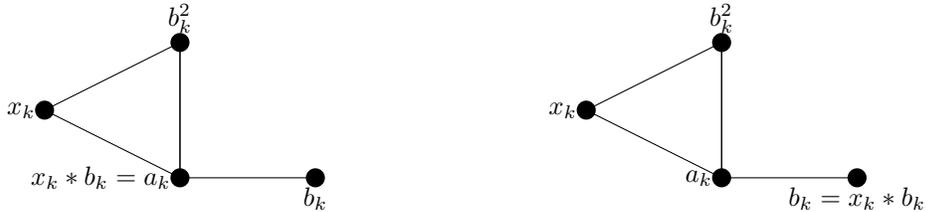
  	\begin{center}
  	\begin{tikzpicture}[scale=1.8]
  		\coordinate (A) at (-2,0.5);
  		\coordinate (B) at (-1,1);
  		\coordinate (C) at (-1,0);
  		\coordinate (D) at (0,0);
  		\draw (A) -- (B) -- (C) -- cycle;
  		
  		\draw (B) -- (C) -- (D) ;
  		\fill (A) circle (2pt);
  		\fill (B) circle (2pt);
  		\fill (C) circle (2pt);
  		\fill (D) circle (2pt);
  		\node[left] at (A) {$x_k$};
  		\node[below][left] at (C) {$x_k*b_k=a_k$};
  		\node[above] at (B) {$b_k^2$};
  		\node[below] at (D) {$b_k$};
  
  	\coordinate (A) at (2,0.5);
  	\coordinate (B) at (3,1);
  	\coordinate (C) at (3,0);
  	\coordinate (D) at (4,0);
  	\draw (A) -- (B) -- (C) -- cycle;
  	
  	\draw (B) -- (C) -- (D) ;
  	\fill (A) circle (2pt);
  	\fill (B) circle (2pt);
  	\fill (C) circle (2pt);
  	\fill (D) circle (2pt);
  	\node[left] at (A) {$x_k$};
  	\node[below][left] at (C) {$a_k$};
  	\node[above] at (B) {$b_k^2$};
  	\node[below] at (D) {$b_k=x_k*b_k$};
  	\end{tikzpicture}
  \end{center}

  If possible, $x_k*b_k\in N(b_k)=\{a_k\}$, \ie $x_k*b_k=a_k$ then we get for all $j \neq k$, $1 \leq j,k \leq n$, $x_k*b_k*a_j=a_k*a_j=0$. Therefore by $(C)$,   $x_k*a_j=0$, for all $j \neq k$, $1 \leq j,k \leq n$. Also, $x_k*a_k=0$. Thus $x_k*a_i=0,$ for all $i, 1\leq i \leq n$. If $x_k \notin \{a_i~|~1\leq i \leq n\}$, then $G[\{x_k,a_i~|~1\leq i \leq n\}] \cong K_{n+1}$, which is a contradiction. 
  Thus, $x_k = a_j$, for some $j \neq k$, $1 \leq j,k \leq n$. Since $a_k=x_k*b_k=a_j*b_k=a_j$, for all $j \neq k$, $1 \leq j,k \leq n$, again a contradiction. 
  
  Now if $x_k*b_k=b_k$, then $0=x_k*b_k^2=x_k*b_k*b_k=b_k*b_k=b_k^2$, a contradiction. 
  
  Thus, there does not exist $x_k\in V(G(S))$ such that $a_k-b_k^2-x_k-a_k$ forms a triangle in $G(S)$. Therefore, we have $b_k^2\perp a_k$ in $G(S)$. Since $b_k$ is a unique element in $G(S)$ such that $b_k\perp a_k$ and hence $b_k^2=b_k$. Thus, $b_i^2=b_i$, for all $i$, $1\leq i\leq n$. \hfill $---(D)$
 
 Further, we claim that $b_1*b_2*\cdots*b_n=0$.\\ Suppose $x=b_1*b_2*\cdots*b_n\neq 0$. Clearly, $x*a_i=0$, for all $i$, $1 \leq i \leq n$. This gives $x \in V(G(S))$. 
 If $x \notin \{a_i~|~1\leq i \leq n\}$, then  $G[\{a_i,x~|~1 \leq i \leq n\}] \cong K_{n+1}$, which is a contradiction. Thus $x =a_k$ for some $k$, $1 \leq k\leq n$. Also by using equation $(D),$ we have $b_k^2=b_k$. Therefore, $0=a_k*b_k=x*b_k=b_1*b_2*\cdots*b_k*\cdots*b_n*b_k=b_1*b_2*\cdots*b_k^2*\cdots*b_n=b_1*b_2*\cdots*b_k*\cdots*b_n=x=a_k$, again a contradiction.  Hence, $x=b_1*b_2*\cdots*b_n=0$. \hfill $---(E)$
 
 \item Now, we claim that $a_i=b_1* b_2*\cdots *b_{i-1}*b_{i+1}*\cdots * b_n,$ for each $i$, $1\leq i\leq n$.

  For each $i$, $1\leq i\leq n$, let us denote  $t_i=b_1* b_2*\cdots *b_{i-1}*b_{i+1}*\cdots * b_n.$\\  
  By using equation $(C)$ repeatedly $(n-1)$ times, we obtain $t_i*a_i=a_i$, for all $1\leq i\leq n$. \\If $t_i=0$, then $a_i=t_i*a_i=0*a_i=0$, which is not possible. Thus $t_i \neq 0$. \\
 Similarly, if $t_i=b_i$, then $a_i=t_i*a_i=b_i*a_i=0$, which is also not possible. \\
 Further, by using $(E)$, we get $t_i*b_i=0,$ for all $i$, $1\leq i\leq n$.\\
 Therefore, we have $t_i\in V(G(S))\setminus \{b_i\} \quad \text{such that} \quad  t_i\in N(b_i)=\{a_i\}.$\\
Thus, $t_i=a_i$, for all $i$,  $1\leq i\leq n$, \ie $a_i=b_1* b_2*\cdots *b_{i-1}*b_{i+1}*\cdots * b_n,$ for each $i$, $1\leq i\leq n$. 
 Hence $a_i^2=a_i*a_i=t_i*a_i=a_i=b_1* b_2*\cdots *b_{i-1}*b_{i+1}*\cdots * b_n$, for all $i$, $1\leq i\leq n$.  \hfill $---(F)$
 
\item Let $X=\{b_i, b_j, \cdots, b_k\}$ and $Y=\{b_l, b_m, \cdots, b_q\}$ be any two distinct, nonempty subsets of $\{b_1,b_2,\cdots,b_n\}$. 
Thus, there exists 
$b_s\in X=\{b_i, b_j, \cdots, b_k\}$ such that $b_s \not \in Y= \{b_l, b_m, \cdots, b_q\}$. 
We claim that $b_i*b_j* \cdots* b_k\neq b_l*b_m*\cdots*b_q$.

Suppose $b_i*b_j* \cdots* b_s* \cdots* b_k= b_l*b_m*\cdots*b_q.$ \\Then $a_s*b_s=0$ and $a_s*b_t=a_s$, for all $t\neq s$, $1\leq t,s\leq n$, this together with $b_s\notin \{b_l, b_m, \cdots, b_q\}$ implies that \[0=a_s*(b_i*b_j*\cdots*b_s*\cdots*b_k)=a_s*(b_l* b_m*\cdots*b_q)=a_s,\] which is not possible. 

Hence  $\{b_i, b_j, \cdots, b_k\}\neq \{b_l, b_m, \cdots, b_q\}$, implies $b_i*b_j* \cdots* b_k\neq b_l*b_m*\cdots*b_q$. \hfill $---(G)$

 \item Now, we denote the sets $A_1, A_2, \cdots, A_{n-1}$ are as follows:

 \[
 \begin{aligned}
 A_1&=\{b_1,b_2,\cdots,b_n\}\\
 A_k&=\{b_i* b_j* \cdots * b_s~|~b_i, b_j, \cdots , b_s~\text{are all} ~k ~ \text{distinct elements in}~ A_1\},\\
 &\quad \quad\text{where} ~ 2\leq k\leq (n-1).
 \end{aligned}
\]

\medskip

Since from equation $(F)$, we have $a_i=b_1* b_2* \cdots* b_{i-1}*b_{i+1}*\cdots*b_{n}$. \\Thus
$A_{n-1}=\{a_i~|1\leq i\leq n\}$.


Clearly, all the sets $A_p, ~1\leq p\leq (n-1)$ are the nonempty subsets in $S$. Further, $A_1$ and $A_{n-1}$ are subsets of $V(G(S))$ such that $A_1 \cap A_{n-1} = \emptyset$.

Now, let $x\in A_k,~2\leq k\leq (n-2)$ be any element. Thus $x=b_i* b_j* \cdots * b_s$, where $b_i, b_j, \cdots , b_s$ are all $k$ distinct elements in $A_1$. Using equation $(E)$ and $(G)$, it is clear that $(x=)b_i* b_j* \cdots * b_s\neq b_1* b_2* \cdots * b_n(=0)$. Also, we have $x*a_i=x*a_j=\cdots=x*a_s=0$, implies that $x\in V(G(S))$. Thus $A_k\subseteq V(G(S))$, for all $k$, $2\leq k\leq (n-2)$. \\Hence $A_p(\neq \emptyset)\subseteq V(G(S))$, for all $p$, $1\leq p\leq (n-1)$. Therefore, we have $\displaystyle\bigcup_{p=1}^{n-1}A_p \subseteq V(G(S))$. Using equation $(D), (E)$ and $(G)$, it is easy to prove that $A_p\cap A_q=\emptyset$, for all $p\neq q$, $1\leq p, q\leq (n-1)$.

Thus, $\{A_k~|~1\leq k\leq n-1\}$ forms a nonempty, disjoint family of subsets in $V(G(S))$. Further, from equation $(D)$, we obtain that $x^2=x$, for every $x\in \displaystyle\bigcup_{p=1}^{n-1}A_p$.

\item Now, we claim that for all $x \in \displaystyle \bigcup_{p=1}^{n-1}{A_p}$, there exist unique $y \in \displaystyle \bigcup_{p=1}^{n-1}{A_p}$ such that $x \perp y$. 

As we have, $b_i\in A_1$ and $a_i\in A_{n-1}$ with $a_i\perp b_i$, it is clear that every element of $A_1\cup A_{n-1}$ has a unique complement in $A_1\cup A_{n-1}$. 

Let $x\in A_k$, where $2\leq k\leq (n-2)$. Thus, $x=b_i*b_j*\cdots*b_s$, where $b_i, b_j, \cdots , b_s$ are all $k$ distinct elements in $A_1$. 

Let $X=\{b_i, b_j, \cdots, b_s\}\subseteq A_1$ with $|X|=k$. Consider $Y=\{b_l, b_m,\cdots,b_q\}\subseteq A_1\setminus X$ with $|Y|=(n-k)$ and $y=b_l*b_m*\cdots*b_q\in A_{n-k}$. Note that, $X\cap Y=\emptyset$ and $X\cup Y=A_1$. From equation $(E)$ and $(G)$, it is easy to observe that $x$ and $y$ are distinct non-zero elements such that $x*y=b_1*b_2*\cdots*b_n=0$. Thus $x, y\in V(G(S))$ such that $x$ and $y$ are adjacent in $G(S)$. 

Suppose $x-y$ is an edge of a triangle, say $x-y-d-x$ in $G(S)$. 
\vspace{0.1cm}

Using equation $(C)$, it is clear that,
$x*a_t=a_t, \quad \text{if}\quad t\in \{l,m,\cdots, q\}\quad \text{and}\quad y*a_t=a_t, \quad \text{if} \quad t\in \{i,j,\cdots, s\}.$
Thus, for any $t\in \{l,m,\cdots, q\}\cup \{i,j,\cdots, s\}=\{1,2,\cdots, n\}$, we have either $x*a_t=a_t$ or $y*a_t=a_t$. This gives that either $x$ is not adjacent to $a_t$ or $y$ is not adjacent to $a_t$. Therefore, the edge $x-y$ does not forms a triangle with any vertex of $A_{n-1}$. Hence $d\notin A_{n-1}$.

Since $d$ is adjacent to both $x$ and $y$, we have $d*x=d*y=0$.  Thus, $d*a_t=d*x*a_t=0$, if $t\in \{l,m,\cdots, q\}$  and $d*a_t=d*y*a_t=0$, if $t\in \{i,j,\cdots, s\}$. Hence $d*a_t=0$, for all $t\in \{1,2,\cdots, n\}$. This implies that $A_{n-1}\cup \{d\}$ forms a clique with $n+1$ vertices in $G(S)$, which contradicts the fact that $\omega(G(S))=n$.
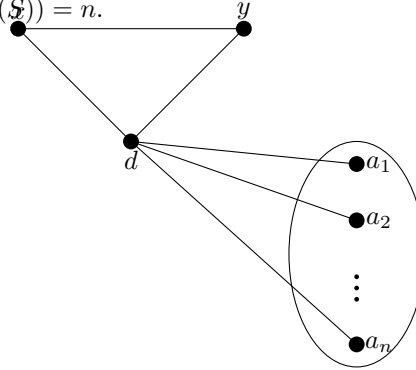
\vspace{-.4in}
\begin{center}
	\begin{tikzpicture}[scale=1.5]
		
		\coordinate (X) at (-1,1);
		\coordinate (Y) at (1,1);
		\coordinate (D) at (0,0);
		
		\draw (X) -- (Y) -- (D) -- cycle;
		
		\draw (2,-1) ellipse (0.6 and 1);
		\node at (2.2,-0.2) {$a_1$};
		\node at (2.2,-0.7) {$a_2$};
		\node at (2.2,-1.8) {$a_n$};
		
		\draw (D) -- (2,-0.2);
		\draw (D) -- (2,-0.7);
		\draw (D) -- (2,-1.8);
		
		\node[above] at (X) {$x$};
		\node[above] at (Y) {$y$};
		\node[below] at (D) {$d$};
		
		\fill (X) circle (2pt);
		\fill (Y) circle (2pt);
		\fill (D) circle (2pt);
		\fill (2,-0.2) circle (2pt);
		\fill (2,-0.7) circle (2pt);
		\fill (2,-1.2) circle (.6pt);
		\fill (2,-1.3) circle (.6pt);
		\fill (2,-1.4) circle (.6pt);
		\fill (2,-1.8) circle (2pt);
	\end{tikzpicture}
\end{center}

 Hence, the edge $x-y$ does not forms a triangle with any vertices in $G(S)$. Thus, $x\perp y$.  Therefore, for all $x \in A_k$, there exists $y \in A_{n-k}$ such that $x \perp y$, where $2 \leq k \leq (n-2)$. Using the assumption that, every vertex has a unique complement in $G(S)$, we get the uniqueness of $y$. 

Thus, for any $1\leq p\leq n-1$, if $x\in A_p$, then there exists unique $y\in A_{n-p}$ such that $x\perp y$. That is, for all $x \in \displaystyle \bigcup_{p=1}^{n-1}{A_p}$, there exists unique $y \in \displaystyle \bigcup_{p=1}^{n-1}{A_p}$ such that $x \perp y$.  \hfill $---(H)$

     \item Now, we claim that $deg(x)\geq 2^k-1$, for every $x\in A_k$, where $2\leq k\leq (n-1)$.

    First, we denote $y=x^{\perp}$ whenever $x\perp y$ in $G(S)$, which is possible due to the fact that the graph $G(S)$ is a complemented and every vertex has the unique complement in $G(S)$.
    
    Now, let $x\in A_k$ be any element, where $2\leq k\leq (n-1)$. Thus, $x=b_i* b_j* \cdots * b_s$, where $b_i, b_j, \cdots , b_s$ are all $k$ distinct elements in $A_1$.

     For simplicity, let $x=b_1*b_2*\cdots *b_k\in A_k$, where $2\leq k\leq (n-1)$. Then by $(H)$, there exists unique $y=b_{k+1}*b_{k+2}*\cdots*b_{n} \in A_{n-k}$ such that $x\perp y$, \ie $y=x^{\perp}=b_{k+1}*b_{k+2}*\cdots*b_{n}$.

     
     Further, we denote
     \[
     \begin{aligned}
     	N_0(x)&=A_{n-k}\cap N(x)=\{x^{\perp}\},\\
     	N_1(x)&=A_{n-k+1}\cap N(x)=\{x^{\perp}*b_i~|~1\leq i\leq k\},\\
     N_2(x)&=A_{n-k+2}\cap N(x)=\{x^{\perp}*b_i*b_j~|~1\leq i\neq j\leq k\}\\
     	&\vdots \\
     	N_{k-2}(x)&=A_{n-k+k-2}\cap N(x)=A_{n-2}\cap N(x)\\&=\{x^{\perp}*\underbrace{b_i*b_j*\cdots *b_s}_{(k-2)~\text{distinct terms }}~|~1\leq i, j,\cdots, s\leq k\}\\  	   N_{k-1}(x)&=A_{n-k+k-1}\cap N(x)=A_{n-1}\cap N(x)\\     	&=\{x^{\perp}*\underbrace{b_i*b_j*\cdots *b_p}_{(k-1)~\text{distinct terms}}~|~1\leq i, j,\cdots, p\leq k\}\\
     	& =\{a_i~|~1\leq i\leq k\}.     \end{aligned}  \]
         Clearly, $\displaystyle\bigcup_{i=0}^{k-1}N_i(x)\subseteq N(x)$.

     Further, we have $N_i(x)\subseteq A_{n-k+i}$, for every $i$, $0\leq i\leq (k-1)$, and $A_{n-k+p}\cap A_{n-k+q}=\emptyset$, for every $p\neq q$, $0\leq p,q\leq (k-1)$, which implies that $N_p(x)\cap N_q(x)=\emptyset$, for every $p\neq q$, $0\leq p,q\leq (k-1)$.
  
     Thus, $|N(x)|\geq |\displaystyle\bigcup_{i=0}^{k-1}N_i(x)|=\displaystyle\sum_{i=0}^{k-1}|N_i(x)|=\displaystyle\sum_{i=0}^{k-1}{}^{k}C_{i}=2^{k}-1$.
     
     Hence $deg(x)=|N(x)|\geq 2^{k}-1$, for every $x\in A_k$, where $2\leq k\leq (n-1)$.
        \medskip
        
    In particular, $deg(x)\geq 2^{n-1}-1$, for every $x\in A_{n-1}$, 
    
    \textit{i.e.}, $deg(a_i)\geq 2^{n-1}-1$, for every $1\leq i\leq n$.

\item Now, we claim that  $deg(a_i)=2^{n-1}-1$, for every $i$, $1\leq i\leq n$. 
First, we prove the claim for the index $i=1$, \ie for $a_1$; then, along similar lines, we can prove the claim for the remaining indices $2\leq i\leq n$.

Since $a_1\perp b_1$, we have $a_1^{\perp}=b_1$. Thus we have the following notations;\vspace{-0.01cm}
 \[
\begin{aligned}
 N_0(a_1)&=A_1\cap N(a_1)=\{b_1\} \\ N_1(a_1)&=A_2\cap N(a_1)=\{b_1*b_i~|~2\leq i\leq n\}\\ 
 N_2(a_1)&=A_3\cap N(a_1)=\{b_1*b_i*b_j~|~2\leq i\neq j\leq n\}\\ 
&\vdots\\
N_{n-3}(a_1)&=A_{n-2}\cap N(a_1)\\
&=\{b_1*\underbrace{b_i*b_j*\cdots *b_s}_{(n-3)~\text{distinct terms}}~|~2\leq i,j,\cdots, s\leq n\}\\
N_{n-2}(a_1)&=A_{n-1}\cap N(a_1)&\\
&=\{b_1*\underbrace{b_i*b_j*\cdots *b_p}_{(n-2)~\text{distinct terms}}~|~2\leq i, j,\cdots, p\leq n\}\\
&=\{a_2, a_3,\cdots, a_n\}. 
\end{aligned}\]

Clearly, $\displaystyle\bigcup_{i=0}^{n-2}N_i(a_1)\subseteq N(a_1)$. 

Suppose $\displaystyle\bigcup_{i=0}^{n-2}N_i(a_1)\subsetneqq N(a_1)$, \ie there exists $x \in N(a_1) \setminus \displaystyle \bigcup_{i=0}^{n-2}{N_i(a_1)}$.  By using an equation $(F)$, it is easy to observe that $x\neq a_1$.   Since $x \in V(G(S))$ and $G(S)$ is complemented, there exists $y\in V(G(S))$ such that $x\perp y$. Further, $x\in N(a_1)$, this imples that $y \notin N(a_1)$.

\medskip
Now, we show that $x,y\not\in \displaystyle \bigcup_{k=1}^{n-1}{A_k}$, with $A_{n-1} \subseteq N(x) \cup N(y)$. \\
Suppose $x \in A_k$ for some $k$, $1 \leq k \leq n-1$. \\Since $x \in N(a_1)$, it follows that
$x \in \displaystyle N_{k-1}(a_1) \subseteq \bigcup_{i=0}^{n-2} N_i(a_1),$
which is a contradiction. Therefore, $x \notin A_k$ for every $k$, $1 \leq k \leq n-1$, and hence
$x \notin \displaystyle\bigcup_{k=1}^{n-1} A_k$.

\medskip

If $y \in \displaystyle \bigcup_{k=1}^{n-1}{A_k}$, then by $(H)$, we have $ y^\perp \in\displaystyle \bigcup_{k=1}^{n-1}{A_k}$. Since every vertex has a unique complement in $G(S)$ and $x\perp y$, it follows that $y^{\perp}=x$. Hence $ x \in\displaystyle \bigcup_{k=1}^{n-1}{A_k}$, which is again a contradiction. Thus $y \notin \displaystyle \bigcup_{k=1}^{n-1}{A_k}$. 

\medskip

Further, if $A_{n-1}\subseteq N(x)$, \ie $x \in \displaystyle \bigcap_{j=1}^{n}{N(a_j)}$, then $G[\{x\} \cup A_{n-1}] \simeq K_{n+1}$, which contradicts to $\omega(G(S))=n\geq 3$. \\

Hence $A_{n-1}\not\subseteq N(x)$, \ie $x \not \in \displaystyle \bigcap_{j=1}^{n}{N(a_j)}.$ This together with $x\in N(a_1)$, we have $x \notin \displaystyle \bigcap_{j=2}^{n}{N(a_j)}$, \ie  $x \notin N(a_j)$, for some $j$, $2 \leq j \leq n$.  

\medskip

 For simplicity, let $j=2$, \ie $x \notin N(a_2)$. \\Suppose $y \notin N(a_2)$, then by Theorem \ref{2.1}, there exists a vertex $z\in V(G(S))$ such that the set $$\{x,a_j,b_2 ~|~ 1\leq j \leq n, j\neq 2\} \subseteq N(y) \cup N(a_2) \subseteq \overline{N(z)}=N(z) \cup \{z\}.$$

Since $b_2\in \overline{N(z)}=N(z) \cup \{z\}$, we have either $b_2=z$ or $b_2\in N(z)$.
Clearly, $b_2\neq z$, since $deg(b_2)=1$. Further, $b_2\in N(z)$ implies that $z=a_2$. Then $x\in \overline{N(z)}=\overline{N(a_2)}$, which contradicts to either the fact that $x\not \in N(a_2)$ or $x \notin \displaystyle\bigcup_{k=1}^{n-1} A_k$.
Thus, $y\in N(a_2)$, \ie $y$ is adjacent to $a_2$. 

On the similar way, we can easily prove that $y\in N(a_j)$, wherever $x\notin N(a_j)$, for all $j$, $1\leq j\leq n$. Thus, we have $A_{n-1} \subseteq N(x) \cup N(y)$. 

 Note that, $a_1\in N(x)$, $a_2\in N(y)$ and $N(x)\cap N(y)=\emptyset$, since $x \perp y$. Hence $deg(x)\geq 2$ and $deg(y)\geq 2$. Furthermore, if the set $(A_{n-1}\setminus \{a_2\})\cup \{x\}$ forms a clique of $n$ elements in $G(S)$, then, by applying similar technique as in observation $(4)$, we can easily prove that $deg(y)=1$, (similar to $deg(b_2)=1$), which contradicts the fact that $deg(y)\geq 2$. Similarly, if $(A_{n-1}\setminus \{a_1\})\cup \{y\}$ forms a clique of $n$ elements in $G(S)$, then, by applying similar technique as in observation $(4)$, we can easily prove that $deg(x)=1$, (similar to $deg(b_1)=1$), which contradicts the fact that $deg(x)\geq 2$. Thus, $G[(A_{n-1}\setminus \{a_2\})\cup \{x\}]\not \cong K_{n}$ and  $G[(A_{n-1}\setminus \{a_1\})\cup \{y\}]\not \cong K_{n}$. But we have 
 $G[(A_{n-1}\setminus \{a_2\})]\cong G[(A_{n-1}\setminus \{a_1\})]\cong K_{n-1}$. Hence $x$ is not adjacent to $a_p$, for some  $a_p\in (A_{n-1}\setminus \{a_2\})$ and $y$ is not adjacent to $a_q$, for some  $a_q\in (A_{n-1}\setminus \{a_1\})$. Clearly, $a_1, a_2, a_p$ and $a_q$ are all distinct elements in $A_{n-1}$ such that $a_1,a_p\in N(x)$ where as $a_2, a_q\in N(y)$.
 
  Without loss of generality, we assume that $a_p=a_3$ and $a_q=a_4$ and for any $k$, $5 \leq k \leq n$, $\{a_1, a_3, a_i ~|~ 5 \leq i \leq k\} \subseteq N(x)$ and $\{a_2, a_4, a_j~|~ (k+1) \leq j \leq n\} \subseteq N(y)$.

Further, for any $k, 5\leq k\leq n$, we look up the elements $b_1*b_3*b_5*b_6*\cdots*b_k$ and $b_2*b_4*b_{k+1}*b_{k+2}*\cdots*b_n$, 
Clearly, $b_1*b_3*b_5*b_6*\cdots*b_k \in A_{k-2}$ and $b_2*b_4*b_{k+1}*b_{k+2}*\cdots*b_n \in A_{n-k+2}$, whereas $x,y \not \in \displaystyle \bigcup_{k=1}^{n-1}{A_k}$. Thus $x,y\not \in \{b_1*b_3*b_5*b_6*\cdots*b_k, b_2*b_4*b_{k+1}*b_{k+2}*\cdots*b_n\}$. 

Now, once we prove that $x \perp (b_2*b_4*b_{k+1}*b_{k+2}*\cdots*b_n)$ and $y \perp (b_1*b_3*b_5*b_6*\cdots*b_k)$, then by using the assumption that every vertex has a unique complement in $G(S)$, we have, $x = (b_2*b_4*b_{k+1}*b_{k+2}*\cdots*b_n)$ and $y = (b_1*b_3*b_5*b_6*\cdots*b_k)$, which is a contradiction. This gives that there does not exists $x \in N(a_1) \setminus \displaystyle \bigcup_{i=0}^{n-2}{N_i(a_1)}$. Therefore, $ \displaystyle\bigcup_{i=0}^{n-2}N_i(a_1)= N(a_1)$ and hence $deg(a_1)=2^{n-1}-1$. 

Now, we prove that $x \perp (b_2*b_4*b_{k+1}*b_{k+2}*\cdots*b_n)$ and $y \perp (b_1*b_3*b_5*b_6*\cdots*b_k)$.

Suppose $x \notin N(b_2*b_4*b_{k+1}*b_{k+2}*\cdots*b_n)$,  then the set $$\{b_1*b_3*b_5*b_6*\cdots*b_k, y, a_i, ~|~1\leq i\leq n\} \subseteq N(b_2*b_4*b_{k+1}*b_{k+2}*\cdots*b_n) \cup N(x).$$ Applying Theorem \ref{2.1}, there exists a vertex $z\in V(G(S))$ such that $$N(b_2*b_4*b_{k+1}*b_{k+2}*\cdots*b_n) \cup N(x)\subseteq \overline{N(z)}=N(z) \cup \{z\}.$$

Therefore, $$\{b_1*b_3*b_5*b_6*\cdots*b_k, y, a_i, ~|~1\leq i\leq n\}\subseteq \overline{N(z)}=N(z) \cup \{z\}.$$

 Then, we have the following possibilities:
\begin{itemize}
\item If $z\in \{a_1, a_3,a_i~|~5\leq i\leq k\}\subseteq N(x)$, then $z \in N(x)\cap N(y)$, a contradiction, since $x \perp y$.
\item If $z\in \{a_2, a_4,a_j~|~(k+1)\leq j\leq n\}$, then by equation $(C)$,  $(b_1*b_3*b_5*b_6*\cdots*b_k)*z=z$, a contradiction to $(b_1*b_3*b_5*b_6*\cdots*b_k)\in \overline{N(z)}$.

\item If $z\notin \{a_i~|~1 \leq i \leq n\}$, \ie $z \notin A_{n-1}$, then $\{z\} \cup A_{n-1} \cong K_{n+1}$, which contradicts to  $\omega(G(S))=n\geq 3$.
\end{itemize}
Thus, $x \in N(b_2*b_4*b_{k+1}*b_{k+2}\cdots*b_n)$.

Similarly, suppose $y \notin N(b_1*b_3*b_5*b_6*\cdots*b_k)$. Then the set 
$$\{b_2*b_4*b_{k+1}*b_{k+2}*\cdots*b_n, x, a_i, ~|~1\leq i\leq n\} \subseteq N(b_1*b_3*b_5*b_6*\cdots*b_k) \cup N(y).$$
Applying Theorem \ref{2.1}, there exists  a vertex $z\in V(G(S))$ such that  $$N(b_1*b_3*b_5*b_6*\cdots*b_k)\cup N(y)\subseteq ~\overline{N(z)} =N(z)\cup \{z\}.$$ 
Therefore, $$\{b_2*b_4*b_{k+1}*b_{k+2}*\cdots*b_n, x, a_i, ~|~1\leq i\leq n\}\subseteq \overline{N(z)}=N(z)\cup \{z\}.$$
Thus, we have the following possibilities:

\begin{itemize}
	\item If $z\in \{a_1, a_3,a_i~|~5\leq i\leq k\}\subseteq N(x)$, then by eqaution $(C)$, $(b_2*b_4*b_{k+1}*b_{k+2}*\cdots*b_n)*z=z$ a contradiction to $(b_2*b_4*b_{k+1}*b_{k+2}*\cdots*b_n)\in \overline{N(z)}$.
	\item If $z\in \{a_2, a_4,a_j~|~(k+1)\leq j\leq n\}\subseteq N(y)$, then $z\in N(x)\cap N(y)$, a contradiction, since, $x\perp y$.
	
	\item If $z\notin \{a_i~|~1 \leq i \leq n\}$, \ie $z \notin A_{n-1}$, then $\{z\} \cup A_{n-1} \cong K_{n+1}$, which contradicts to  $\omega(G(S))=n\geq 3$.
\end{itemize}

Thus, $y \in N(b_1*b_3*b_5*b_6*\cdots*b_k)$.\\

By claim $(10)$, we have $(b_1*b_3*b_5*b_6*\cdots*b_k) \perp (b_2*b_4*b_{k+1}*b_{k+2}*\cdots*b_n)$. Thus we get $x-(b_2*b_4*b_{k+1}*b_{k+2}\cdots*b_n)-(b_1*b_3*b_5*b_6*\cdots*b_k)-y-x$ a path in $G(S)$, where $x\perp y.$

 Therefore, the edges $x - (b_2*b_4*b_{k+1}*b_{k+2}*\cdots*b_n)$ and $y-(b_1*b_3*b_5*b_6*\cdots*b_k)$ are the edges of a triangle, say $s-x-(b_2*b_4*b_{k+1}*b_{k+2}*\cdots*b_n)-s$ and $t - y - (b_1*b_3*b_5*b_6*\cdots*b_k) - t$ in $G(S)$. 

\medskip

If $s=t$, the $x-y-t(=s)-x$ forms a triangle in $G(S)$, which is not possible, since $x\perp y$. 

Thus, $s$ and $t$ are distinct vertices in $G(S)$ with $t\notin N(x)$ and $s\notin N(y)$.

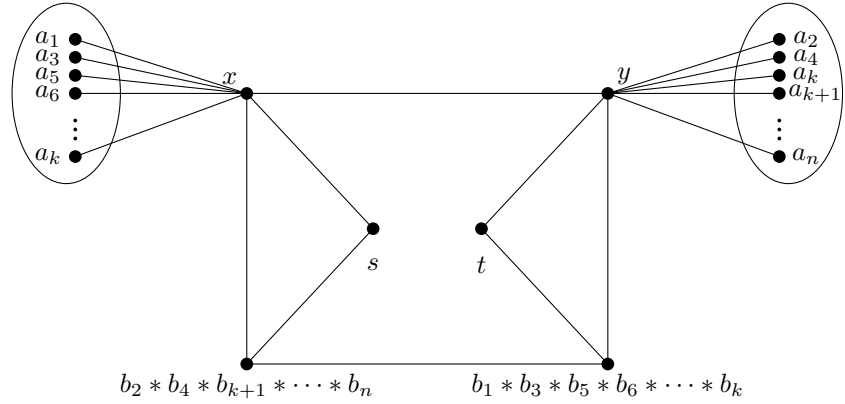
\begin{center}
	\begin{tikzpicture}[scale=1.2]
		
		\draw (0,0) rectangle (4,3);
		
		\draw (0,3) -- (1.4,1.5);
		\draw (0,0) -- (1.4,1.5);
		\draw (4,3) -- (2.6,1.5);
		\draw (4,0) -- (2.6,1.5);
		\node at (1.4,1.1) {$s$};
		\node at (2.6,1.1) {$t$};
		\node[above left] at (0,3) {$x$};
		\node[above right] at (4,3) {$y$};
		\node[below] at (0,0) {$b_2*b_4*b_{k+1}*\cdots*b_n$};
		\node[below] at (4,0) {$b_1*b_3*b_5*b_6*\cdots*b_k$};
		
		\fill (0,3) circle (2pt);
		\fill (4,3) circle (2pt);
		\fill (0,0) circle (2pt);
		\fill (4,0) circle (2pt);
		\fill (1.4,1.5) circle (2pt);
		\fill (2.6,1.5) circle (2pt);
		
		\draw (-2,3) ellipse (0.6 and 1);
		\node at (-2.2,3.6) {$a_1$};
		\node at (-2.2,3.4) {$a_3$};
		\node at (-2.2,3.2) {$a_5$};
		\node at (-2.2,3.0) {$a_6$};
		\node at (-2.2,2.3) {$a_k$};
		
		\draw (0,3) -- (-1.9,3.6);
		\draw (0,3) -- (-1.9,3.4);
		\draw (0,3) -- (-1.9,3.2);
		\draw (0,3) -- (-1.9,3.0);
		\draw (0,3) -- (-1.9,2.3);

		\fill (-1.9,3.6) circle (2pt);
		\fill (-1.9,3.4) circle (2pt);
		\fill (-1.9,3.2) circle (2pt);
		\fill (-1.9,3.0) circle (2pt);
		\fill (-1.9,2.7) circle (.6pt);
		\fill (-1.9,2.6) circle (.6pt);
		\fill (-1.9,2.5) circle (.6pt);
		\fill (-1.9,2.3) circle (2pt);
		\draw (6,3) ellipse (0.6 and 1);
		\node at (6.2,3.6) {$a_2$};
		\node at (6.2,3.4) {$a_4$};
		\node at (6.2,3.2) {$a_k$};
		\node at (6.3,3.0) {$a_{k+1}$};
		\node at (6.2,2.3) {$a_n$};
		
		\draw (4,3) -- (5.9,3.6);
		\draw (4,3) -- (5.9,3.4);
		\draw (4,3) -- (5.9,3.2);
		\draw (4,3) -- (5.9,3.0);
		\draw (4,3) -- (5.9,2.3);

		\fill (5.9,3.6) circle (2pt);
		\fill (5.9,3.4) circle (2pt);
		\fill (5.9,3.2) circle (2pt);
		\fill (5.9,3.0) circle (2pt);
		\fill (5.9,2.7) circle (.6pt);
		\fill (5.9,2.6) circle (.6pt);
		\fill (5.9,2.5) circle (.6pt);
		\fill (5.9,2.3) circle (2pt);
	\end{tikzpicture}
\end{center}


Since  $t\not\in N(x),$ \ie $x$ and $t$ are not adjacent, then by Theorem \ref{2.1}, there exists $z\in V(G(S))$ such that $$\{s,y,b_2*b_4*b_{k+1}*b_{k+2}*\cdots*b_n, b_1*b_3*b_5*b_6*\cdots*b_k \}\subseteq N(x) \cup N(t)\subseteq \overline{N(z)}=N(z)\cup \{z\}.$$

Consider the following possibilities:
\begin{itemize}\item
 If $z=b_2*b_4*b_{k+1}*b_{k+2}*\cdots*b_n$, then $y\in N(b_2*b_4*b_{k+1}*b_{k+2}*\cdots*b_n)$, which is not possible.
 \item If $z=b_1*b_3*b_5*b_6*\cdots*b_k$, then $s\in N(b_1*b_3*b_5*b_6*\cdots*b_k)$, which is also not possible.
 \item  If $z\notin \{b_1*b_3*b_5*b_6*\cdots*b_k, ~ b_2*b_4*b_{k+1}*b_{k+2}*\cdots*b_n\}$, then we have, $z ~-~ (b_1*b_3*b_5*b_6*\cdots*b_k)~ - ~(b_2*b_4*b_{k+1}*b_{k+2}*\cdots*b_n)~ -~ z$ forms a triangle in $G(S)$, which is a contradiction to $(b_1*b_3*b_5*b_6*\cdots*b_k)~ \perp ~(b_2*b_4*b_{k+1}*b_{k+2}*\cdots*b_n)$.

    \end{itemize}

\medskip

Further, we have $s\not\in N(y),$ \ie $y$ and $s$ are not adjacent in $G(S)$. Therefore, applying Theorem \ref{2.1}, there exists $z\in V(G(S))$ such that $$\{x,t,b_2*b_4*b_{k+1}*b_{k+2}*\cdots*b_n, b_1*b_3*b_5*b_6*\cdots*b_k \}\subseteq N(y) \cup N(s)\subseteq \overline{N(z)}=N(z)\cup \{z\}.$$

Again, we consider the following possibilities: \begin{itemize}\item If $z=b_2*b_4*b_{k+1}*b_{k+2}*\cdots*b_n$, then $t\in N(b_2*b_4*b_{k+1}*b_{k+2}*\cdots*b_n)$, which is not possible.
\item If $z=b_1*b_3*b_5*b_6*\cdots*b_k$, then $x\in N(b_1*b_3*b_5*b_6*\cdots*b_k)$, which is also not possible.
\item If $z\notin \{b_1*b_3*b_5*b_6*\cdots*b_k, ~ b_2*b_4*b_{k+1}*b_{k+2}*\cdots*b_n\}$, then we have, $z ~-~ (b_1*b_3*b_5*b_6*\cdots*b_k)~ - ~(b_2*b_4*b_{k+1}*b_{k+2}*\cdots*b_n)~ -~ z$ forms a triangle in $G(S)$, which is a contradiction to $(b_1*b_3*b_5*b_6*\cdots*b_k)~ \perp ~(b_2*b_4*b_{k+1}*b_{k+2}*\cdots*b_n)$.
\end{itemize}

Hence, there does not exist $s,t \in V(G(S))$ such that $x - s - (b_2*b_4*b_{k+1}*b_{k+2}*\cdots**b_n) - x$ and $y - t - (b_1*b_3*b_5*b_6*\cdots*b_k) - y$ forms a triangles in $G(S)$. Thus, $x \perp (b_2*b_4*b_{k+1}*b_{k+2}*\cdots*b_n)$ and $y \perp (b_1*b_3*b_5*b_6*\cdots*b_k)$. \\Thus, $ \displaystyle\bigcup_{i=0}^{n-2}N_i(a_1)= N(a_1)$ and hence $deg(a_1)=2^{n-1}-1$. 

On similar way, we can prove that $deg(a_i)=2^{n-1}-1$, for all $a_i \in A_{n-1}$.

\item Now, we claim that $deg(x)=2^{k}-1$ for all $x \in A_k$, where $2 \leq k \leq (n-2)$.

Let $x \in A_k$ be any element, where $2 \leq k \leq (n-2)$. Therefore $x=b_i* b_j* \cdots * b_s$, where $b_i, b_j, \cdots , b_s$ are all $k$ distinct elements in $A_1$.

For simplicity, let $x=b_1*b_2*\cdots *b_k\in A_k$. Then  $\displaystyle\bigcup_{i=0}^{k-1}N_i(x)\subseteq N(x)$. \\Suppose $\displaystyle\bigcup_{i=0}^{k-1}N_i(x)\subsetneqq N(x)$, \ie there exists $y\in N(x)\setminus \displaystyle\bigcup_{i=0}^{k-1}N_i(x)$.\\
Since $y\in N(x)$, we have $y*x=y*(b_1*b_2*\cdots *b_k)=0$. Suppose $y*b_1*b_2*\cdots *b_{k-1}\neq 0$. Then $y*b_1*b_2*\cdots *b_{k-1}\in V(G(S))$ such that 
$y*b_1*b_2*\cdots *b_{k-1}\in N(b_k)\cup \{b_k\}$. 

If $y*b_1*b_2*\cdots *b_{k-1}=b_k$, then using equation $(A)$ with $k\neq 1$, we have, $0=y*b_1*b_2*\cdots *b_{k-1}*a_1=b_k*a_1=a_1$, which is a contradiction. 

If $y*b_1*b_2*\cdots *b_{k-1}\in N(b_k)=\{a_k\}$, then $y*b_1*b_2*\cdots *b_{k-1}*a_j=a_k*a_j=0$, for all $j$, $(k+1)\leq j\leq n$.  Again using an equation $(A)$, we get, $y*a_j=0$, for all $j$, $(k+1)\leq j\leq n$. 

Thus, for any $j$, $(k+1)\leq j\leq n$, we have $$y\in N(a_j)=\displaystyle \bigcup_{i=0}^{n-2}N_i(a_j)=\displaystyle \bigcup_{i=0}^{n-2}(A_{i+1}\cap N(a_j)).$$ Hence, $y\in A_{p+1}\cap N(a_j)$, for some $p$, $0\leq p\leq n-2$. \\
If $p=0$, \ie $y\in A_1\cap N(a_j)=\{b_j\}$. Hence $y=b_j$. Then $x\in N(y)=N(b_j)=\{a_j\}\subseteq A_{n-1}$, which contradicts to $x \in A_k$, where $2 \leq k \leq (n-2)$.

Thus $y\in A_{p+1}\cap N(a_j)$, for some $p$, $1\leq p\leq n-2$.

Hence $y=b_j*b_r*\cdots*b_t$, where $b_j,b_r, \cdots,b_t$ are all $p+1$ distinct elements in $A_1$.

Clearly, $0=x*y=b_1*b_2*\cdots*b_k*b_j*b_r*\cdots*b_t$ which implies that $k+p+1\geq n$. This together with $1\leq p\leq n-2$, we have $2\leq n-k\leq p+1\leq n-1$. Thus $y\in A_{p+1}$, for some $n-k\leq p+1\leq n-1$.  Also, we have $y\in N(x)$. Therefore, $y\in A_{p+1}\cap N(x)$, for some $n-k\leq p+1\leq n-1$. This implies that $y\in N_{q}(x)=A_{n-k+q}\cap N(x)$, for some $q$, $0\leq q\leq k-1$.  

Thus, $y\in \displaystyle \bigcup_{i=0}^{k-1}N_i(x)$, again a contradiction. 

Hence $y*b_1*b_2*\cdots*b_{k-1}=0$.

Similarly, if we assume that $y*b_1*b_2*\cdots*b_{k-2}\neq 0$, then $y*b_1*b_2*\cdots*b_{k-2}=b_{k-1}$ or $y*b_1*b_2*\cdots*b_{k-2}=a_{k-1}$. In both the cases, we get a contradiction. 
Thus $y*b_1*b_2*\cdots*b_{k-2}= 0$. 

Continuing in this way, we get, $y*b_1= 0$. Therefore, $y\in N(b_1)\cup \{b_1\}=\{a_1,b_1\}$. 

If $y=b_1$, then $b_1^2=0$, a contradiction. 

If $y=a_1$, then $y\in A_{n-1}\cap N(x)=N_{k-1}(x)\subseteq \displaystyle \bigcup_{i=0}^{k-1}N_{i}(x)$, which is again a contradiction. 

Therefore, there does not exists  $y\in N(x)\setminus \displaystyle\bigcup_{i=0}^{k-1}N_i(x)$. Thus $N(x)= \displaystyle\bigcup_{i=0}^{k-1}N_i(x)$ and \\ $deg(x)=|N(x)|=|\displaystyle\bigcup_{i=0}^{k-1}N_i(x)|=2^{k}-1$, for every $x\in A_k$, where $2\leq k\leq (n-2)$.
\item It is easy to observe that, if $x\in \displaystyle \bigcup_{k=1}^{n-1}A_k$, then $N(x)\subseteq \displaystyle \bigcup_{k=1}^{n-1}A_k$. 
Further, we have  $\displaystyle \bigcup_{k=1}^{n-1}A_k\subseteq V(G(S))$. Now, suppose there exists $y \in V(G(S))$ such that $y \notin \displaystyle \bigcup_{k=1}^{n-1}{A_k}$. Since by Theorem \ref{2.1}, $G(S)$ is connected. Also, we have $\displaystyle \bigcup_{k=1}^{n-1}A_k\neq \emptyset$. Thus, $y$ is adjacent to $x$, for some $x\in \displaystyle \bigcup_{k=1}^{n-1}A_k$. Hence $y\in N(x)\subseteq \displaystyle \bigcup_{k=1}^{n-1}A_k$, a contradiction. Therefore, we have $\displaystyle \bigcup_{k=1}^{n-1}A_k= V(G(S))$. 

In a nutshell, we have proved that the sets $A_k,~ (1\leq k\leq n-1)$,  are nonempty, disjoint subsets of $V(G(S))$ such that $$V(G(S))=\displaystyle\bigcup_{k=1}^{n-1}A_k,$$ \ie the family of sets $\{A_k~|~1\leq k\leq (n-1)\}$ forms a partition of the vertex set $V(G(S))$. Moreover, we have shown that $$x^2=x, \quad \text{for every} \quad x\in \displaystyle\bigcup_{k=1}^{n-1}A_k=V(G(S)) \quad \text{and}$$ $$deg(x)=2^k-1, \quad \text {for every}  \quad x\in A_k ~(1\leq k\leq n-1).$$ 

Therefore, we have $$|V(G(S))|=|\displaystyle\bigcup_{k=1}^{n-1}A_k|=\displaystyle\sum_{k=1}^{n-1}|A_k|=\displaystyle\sum_{k=1}^{n-1}{}^{n}C_{k}=2^n-2.$$

Let $X=\{x_1,x_2, \cdots x_n\}$ be any nonempty set and $\mathcal{P}(X)$ denotes a power set of $X$. As we have $|V(G(\mathcal{P}(X)))|=2^n-2.$
Thus we define a bijective map $f:V(G(S)) \rightarrow V(G(\mathcal{P}(X)))$ as follows;
\[
\begin{aligned}
f(a_i)&=\{x_i\}, \quad \text{for all}~a_i \in A_{n-1},\\ 
f(b_i)&=X \setminus \{x_i\}, \quad \text{for all}~ b_i \in A_1, ~\text{and}\\
f(x)&=X \setminus \{x_i,x_j,\cdots,x_s\},\quad \text{for all}~ x~(=\underbrace{b_i*b_j*\cdots*b_s}_{~~k~\text{distinct terms in}~ A_1}) \in A_k,\\& \quad \text{where} ~~3 \leq k \leq n-2.
\end{aligned}\]
It is easy to verify that the function $f$ is graph homomorphism and hence $f$ is an isomorphism from $G(S)$ to $G(\mathcal{P}(X))$, \ie $G(S)\cong G(\mathcal{P}(X))$.

Now, we assume that $S=V(G(S))\cup \{0,1\}$, where  $0$ and $1$ denotes the zero element and  the unity in $S$ respectively, \ie $x*0=0$ and $x*1=x$ for every $x\in S$. Clearly, $S$ is a commutative idempotent semigroup with $|S|=|\mathcal{P}(X)|$.

We can define the bijective map $T:S\rightarrow \mathcal{P}(X)$ as follows: 
\[
\begin{aligned}
T(x)&=f(x), ~~ \text{for every}~ x\in V(G(S))=S\setminus \{0,1\}, \\ 
\quad &T(0)=\emptyset, ~~\text{and}~~ T(1)=X.
\end{aligned}
\]

Clearly, $T(x*y)=T(x)\cap T(y)$, for every $x,y\in \{0,1\}$. \\Let $x, y\in V(G(S))=\displaystyle\bigcup_{k=1}^{n-1}A_k$ be any elements. Thus, $x\in A_p$ and $y\in A_q$, for some $1\leq p,q\leq n-1$. Hence $x=\underbrace{b_i*b_j*\cdots*b_s}_{~~p~\text{distinct terms in}~ A_1}$ and $y=\underbrace{b_l*b_m*\cdots*b_t}_{~~q~\text{distinct terms in}~ A_1}$. \\Therefore, $$x*y=\underbrace{b_i*b_j*\cdots*b_t}_{~~r ~\text{distinct terms in}~ A_1}, \quad \text{ where}\quad 1\leq r\leq (p+q).$$ Thus, $$T(x*y)= X\setminus \quad \{\underbrace{x_i, x_j, \cdots, x_t}_{r ~\text{distinct terms in}~ X}~|~1\leq r\leq (p+q)\}.$$

Further, we have $$T(x)=X\setminus \quad \{\underbrace{x_i, x_j, \cdots, x_s}_{p ~\text{distinct terms in}~ X}\} \quad \text{and}\quad T(y)=X\setminus \quad \{\underbrace{x_l, x_m, \cdots, x_t}_{q ~\text{distinct terms in}~ X}\}.$$ Clearly, 
\[
\begin{aligned}
T(x)\cap T(y)&=X\setminus \quad \{\{\underbrace{x_i, x_j, \cdots, x_s}_{p ~\text{distinct terms in}~ X}\} \cup \{\underbrace{x_l, x_m, \cdots, x_t}_{q ~\text{distinct terms in}~ X}\}\}\\
&=X\setminus \quad \{\underbrace{x_i, x_j, \cdots, x_t}_{r ~\text{distinct terms in}~ X}~|~1\leq r\leq (p+q)\}\\&=T(x*y). \end{aligned}\] 
 Therefore, $T(x*y)=T(x)\cap T(y)$, for all $x,y\in V(G(S))\cup \{0,1\}=S$.\\ Thus, the bijective map $T$ is a semigroup homomorphism and hence it is an isomorphism from $S$ to $\mathcal{P}(X)$, \ie $S\cong \mathcal{P}(X)$.

\end{enumerate}
\end{proof}


\begin{thebibliography}{99}
\bibitem{AL} D. Anderson and P. Livingstone, {\it The zero-divisor graph of a commutative ring}, J. Algebra, {\bf 217} (1999), 434-447.

\bibitem{AN} D.	Anderson and M. Naseer, {\it Beck’s coloring of a commutative ring},	J. Algebra, {\bf 159} (1993), 500–514.

\bibitem{B} I. Beck,  {\it Coloring of  commutative rings},	J. Algebra, {\bf 116} (1988), 208-226.
	
\bibitem{C} C. Bender, P. Cappaert, R. DeCoste, and L. DeMeyer, {\it Complemented zero-divisor graphs associated with finite commutative semigroups}, Comm. Algebra, \textbf{52(7)} (2024), 2852-2867.
	
\bibitem{DD} F. DeMeyer, L. DeMeyer, {\it Zero-divisor graphs of semigroups}, J. Algebra, {\bf 283} (2005), 190–198.

\bibitem{DMS} F. DeMeyer, T. McKenzie and K. Schneider, {\it The zero-divisor graph of a commutative semigroup}, Semigroup Forum, {\bf 65} (2002), 206-214.

\bibitem{KTJ} A. Khiste, G. Tarte and Vinayak Joshi, {\it Counter example to conjectures on complemented zero-divisor graphs of semigroups}, Bull. Austral. Math. Soc. {\bf 113} (2026), 189-193.

\bibitem{L}	J. LaGrange, {\it Complemented zero divisor graphs and Boolean rings}, J. Algebra, {\bf 315} (2007), 600-611.
	
\bibitem{LW} D.	Lu and T. Wu,  {\it The zero-divisor graphs of partially ordered sets and an application to semigroups}, Graphs Combin., {\bf 26} (2010), 793–804.


\end{thebibliography}
\end{document}